\documentclass[11pt,a4paper]{amsart}

\usepackage{amsmath,amssymb,amsthm,amscd,verbatim, enumerate}
\usepackage{graphicx}
\usepackage{epsfig}
\usepackage{hyperref}

\newtheorem{thm}{Theorem}
\newtheorem{lem}[thm]{Lemma}

\newtheorem{conj}[thm]{Conjecture}

\newtheorem{cor}[thm]{Corollary}
\newtheorem*{thm*}{Theorem}
\newtheorem{quest}[thm]{Question}

\theoremstyle{definition}

\theoremstyle{remark}

\renewcommand{\le}{\leqslant}

\renewcommand{\ge}{\geqslant}

\begin{document}

\title[Equitable partition of graphs into induced forests]{Equitable partition of graphs\\ into induced forests}

\author{Louis Esperet} \address{Laboratoire G-SCOP (CNRS,
   Grenoble-INP), Grenoble, France}
\email{louis.esperet@g-scop.fr}

\author{Laetitia Lemoine} \address{Laboratoire G-SCOP (CNRS,
   Grenoble-INP), Grenoble, France}
\email{laetitia.lemoine@g-scop.fr}

\author{Fr\'ed\'eric Maffray} \address{Laboratoire G-SCOP (CNRS,
   Grenoble-INP), Grenoble, France}
 \email{frederic.maffray@g-scop.fr}

\thanks{The authors are partially supported by ANR Project Stint
  (\textsc{anr-13-bs02-0007}), and LabEx PERSYVAL-Lab
  (\textsc{anr-11-labx-0025}).}

\date{}
\sloppy

\begin{abstract}
An \emph{equitable partition} of a graph $G$ is a partition of the
vertex-set of $G$ such that the sizes of any two parts differ by at
most one. We show that every graph with an acyclic coloring with at
most $k$ colors can be equitably partitioned into $k-1$ induced
forests. We also prove that for any integers $d\ge 1$ and $k\ge 3^{d-1}$,
any $d$-degenerate graph can be equitably partitioned into $k$ induced
forests. 

Each of these results implies the existence of a constant $c$ such
that for any $k \ge c$, any planar graph has an equitable partition
into $k$ induced forests. This was conjectured by Wu, Zhang, and Li in
2013.
\end{abstract}

\maketitle


An \emph{equitable partition} of a graph $G$ is a partition of the
vertex-set of $G$ such that the sizes of any two parts differ by at
most one. Hajnal and Szemer\'edi~\cite{HS70} proved the following
result, which was conjectured by Erd\H{o}s (see also \cite{KKMS10} for a shorter proof).

\begin{thm}\label{th:col}
For any integers $\Delta$ and $k\ge \Delta+1$, any graph with maximum degree $\Delta$ has
an equitable partition into $k$ stable sets.
\end{thm}

Note that there is no constant $c$, such that for any $k\ge c$, any
star can be equitably partitioned into $k$ stable sets.
Wu, Zhang, and Li made the following two conjectures~\cite{WZL13}.

\begin{conj}\label{conj:2}
There is a constant $c$ such that for any $k\ge c$, any planar graph
can be equitably partitioned into $k$ induced forests.
\end{conj}

\begin{conj}\label{conj:1}
For any integers $\Delta$ and $k\ge \lceil \tfrac{\Delta+1}{2}\rceil$,
any graph of maximum degree $\Delta$ can be equitably partitioned into
$k$ induced forests.
\end{conj}

A proper coloring of a graph $G$ is \emph{acyclic} if any cycle of $G$
contains at least 3 colors. We first prove the following result.

\begin{thm}\label{th:acy}
Let $k\ge 2$. If a graph $G$ has an acyclic coloring with at most $k$ colors, then $G$ can be
equitably partitioned into $k-1$ induced forests.
\end{thm}

\begin{proof}
The proof proceeds by induction on $k\ge 2$. If $k=2$ then $G$ itself
is a forest and the result trivially holds, so we can assume that
$k\ge 3$. Let $V_1,\ldots,V_k$ be the color classes in some acyclic
$k$-coloring of $G$ (note that some sets $V_i$ might be empty). Let
$n$ be the number of vertices of $G$. Without loss of generality, we
can assume that $V_1$ contains at most $\tfrac{n}{k}\le
\tfrac{n}{k-1}$ vertices. Observe that the sum of the number of
vertices in $V_1\cup V_i$, $2\le i \le k$, is $n+(k-2)|V_1|\ge n$. It follows
that there exists a color class, say $V_2$, such that $V_1\cup V_2$
contains at least $\tfrac{n}{k-1}$ vertices. Let $S$ be a set of
vertices of $G$ consisting of $V_1$ together with
$\lfloor\tfrac{n}{k-1}\rfloor-|V_1|$ vertices of $V_2$, and let $H$ be the graph
obtained from $G$ by removing the vertices of $S$. Note that $S$
induces a forest in $G$, and $H$ has an acyclic coloring with at most
$k-1$ colors (with color classes $V_2\setminus S,V_3,\ldots,V_k$). By the induction hypothesis, $H$ has an equitable
partition into $k-2$ induced forests, and therefore $G$ has an
equitable partition into $k-1$ induced forests.
\end{proof}

It was proved by Borodin~\cite{Bor79} that any planar graph has an acyclic coloring
with at most 5 colors. Therefore, Theorem~\ref{th:maind} implies
Corollary~\ref{cor:1}, which is a positive answer to
Conjecture~\ref{conj:2}.

\begin{cor}\label{cor:1}
For any $k\ge 4$, any planar graph can be equitably partitioned into $k$ induced forests.
\end{cor}

We now prove stronger results in two different ways. We first show the induced forests can be chosen to be very specific. We then show that graphs from a class that is much wider than the class of planar graphs can also be equitably partitioned into constantly many induced forests.

\medskip

A \emph{star coloring} of a graph $G$ is a proper coloring of the
vertices of $G$ such that any two color classes induce a star forest.
Using the same proof as that of Theorem~\ref{th:acy}, it is easy to show
the following result.

\begin{thm}\label{th:star}
Let $k\ge 2$. If a graph $G$ has a star coloring with at most $k$ colors, then $G$ can be
equitably partitioned into $k-1$ induced star forests.
\end{thm}

It was proved by Albertson \emph{et al.}~\cite{AGKKR04} that every
planar graph has a star coloring with at most 20 colors. The next
corollary follows as an immediate consequence.

\begin{cor}\label{cor:2}
For any $k\ge 19$, any planar graph can be equitably partitioned into
$k$ induced star forests.
\end{cor}

Indeed, if one is not too regarding on the constant, a stronger result
holds. An \emph{orientation} of a graph $G$ is a directed graph
obtained from $G$ by orienting each edge in either of two possible
directions. An \emph{out-star} (resp. \emph{in-star}) is the
orientation of a star such that every edge is oriented from the center
of the star to the leaf (resp. from the leaf to the center of the
star).

\begin{thm}\label{th:ori}
For any $k\ge 319$, any orientation of a planar graph can be equitably partitioned into
$k$ induced forests of in- and out-stars.
\end{thm}

\begin{proof}
It was proved by Raspaud and Sopena~\cite{RS94} that every orientation
of a planar graph $G$ has an acyclic coloring with at most 80 colors such
that for any two colors classes $V_i$ and $V_j$, if there is an arc
$(u,v)$ with $u\in V_i$ and $v\in V_j$, then there is no arc $(x,y)$
with $y\in V_i$ and $x\in V_j$.

It was proved in~\cite{AGKKR04} that if all the minors of some graph
$G$ are $r$-colorable, then any acyclic coloring of $G$ with $k$
colors can be refined into a star coloring of $G$ with $rk$ colors. By
the Four Color Theorem, all the minors of a planar graph are
4-colorable, and therefore the acyclic coloring of $G$ with 80 colors
mentioned above can be refined into a star coloring, with the same
additional property, using no more than $80\cdot4=320$ colors. In
particular, every two color classes induce a forest of in- and
out-stars. The remainder of the proof follows the same lines as the
proofs of Theorems~\ref{th:acy} and~\ref{th:star}.
\end{proof}

It is known that graphs with bounded acyclic chromatic number also have bounded star chromatic number~\cite{AGKKR04} and bounded oriented chromatic number~\cite{RS94}, so it follows that the results of Theorems~\ref{th:star} and~\ref{th:ori} hold for any class of graphs with bounded acyclic chromatic number (with possibly larger constants).

\medskip

A graph $G$ is \emph{$d$-degenerate} if every subgraph of $G$ contains
a vertex of degree at most $d$. In the remainder of this article, we prove the
following result. 

\begin{thm}\label{th:maind}
For any integers $d\ge 1$ and $k\ge 3^{d-1}$, any $d$-degenerate graph
can be equitably partitioned into $k$ induced forests.
\end{thm}

It follows from Euler's formula that every planar graph is
5-degenerate. Therefore, Theorem~\ref{th:maind} also implies
Conjecture~\ref{conj:2} (with $c=81$ instead of $c=4$ in
Corollary~\ref{cor:1}). For a graph $G$, let $\chi_a(G)$ denote the
least integer $k$ such that $G$ has an acyclic coloring with $k$
colors. It is known that there is a function $f$ such that every graph
$G$ is $f(\chi_a(G))$-degenerate~\cite{Dvo07}. However there exist families of
2-degenerate graphs with unbounded acyclic chromatic number. It
follows that Theorem~\ref{th:maind} can be applied to wider classes of
graphs than Theorems~\ref{th:acy},~\ref{th:star}, and~\ref{th:ori}.

\smallskip

A class of graphs is \emph{hereditary} if it is closed under taking
induced subgraphs. 

\begin{lem}\label{lem:div}
  Let $\ell$ be an integer and $\mathcal C$ be a hereditary class of
  graphs such that every graph in ${\mathcal C}$ can be equitably
  partitioned into $\ell$ induced forests. Then for any $k \ge \ell$,
  any graph in ${\mathcal C}$ can be equitably partitioned into $k$
  induced forests.
\end{lem}

\begin{proof}
Let $G$ be a graph of ${\mathcal C}$, and let $n$ be the number of
vertices of $G$. Let $n=kq+s$, with $0\le s<k$. Note that an equitable
partition of $G$ into $k$ sets consists of $s$ sets of size $\lceil
n/k \rceil$ and $k-s$ sets of size $\lfloor n/k \rfloor$.

 Let $G_0=G$. For any $1\le i \le k-\ell$, we inductively define $G_i$
 as a graph obtained from $G_{i-1}$ by removing a set $S_{i-1}$ of
 $\lceil n/k \rceil$ vertices (if $i\le s$) or $\lfloor n/k \rfloor$
 vertices (otherwise) inducing a forest in $G_{i-1}$. The existence of
 such an induced forest follows from the fact that for any $n'\ge
 \tfrac \ell k \,n$, any induced subgraph of $G$ on $n'$ vertices
 contains an induced forest on at least $\lceil n'/\ell \rceil \ge
 \lceil n/k \rceil$ vertices.
 By assumption, the graph $G_{k-\ell}$ can be equitably partitioned
into $\ell$ induced forests (each on $\lfloor n/k \rfloor$ or $\lceil
n/k \rceil$ vertices). Combining these induced forests with
$S_0,S_1,\ldots, S_{k-\ell-1}$, we obtain an equitable partition of
$G$ into $k$ induced forests.
\end{proof}

The following result was proved in~\cite{KNP05}.

\begin{thm}\label{th:kdeg}
Let $k\ge 3$ and $d\ge 2$. Then every $d$-degenerate graph can be
equitably partitioned into $k$ $(d - 1)$-degenerate graphs.
\end{thm}

We now give a short proof of Theorem~\ref{th:maind} using
Lemma~\ref{lem:div} and Theorem~\ref{th:kdeg}.

\bigskip

\noindent \emph{Proof of Theorem~\ref{th:maind}.}  By
Lemma~\ref{lem:div}, it is enough to show that any $d$-degenerate
graph has an equitable partition into $3^{d-1}$ induced forests.

We prove this result
by induction on $d\ge 1$.  If $d=1$, the result follows from the fact
that a 1-degenerate graph is a forest. Assume that $d\ge 2$. By
Theorem~\ref{th:kdeg}, $G$ has an equitable partition into three $(d
- 1)$-degenerate graphs. By the induction, each of these graphs has an
equitable partition into $3^{d-2}$ induced forests, therefore $G$ has
an equitable partition into $3\cdot 3^{d-2}=3^{d-1}$ induced forests.\hfill $\Box$

\bigskip

\noindent {\bf Open problems.} It remains to determine whether every
planar graph has an equitable partition into three induced forests (partial results on this problem can be found in~\cite{Zh15}). By
Theorems~\ref{th:acy} and~\ref{th:maind}, a possible counterexample
must have acyclic chromatic number equal to 5 and cannot be
2-degenerate.

\smallskip

It was proved by Poh~\cite{Poh90} that every planar graph has a partition into three induced
\emph{linear forests} (i.e. graphs in which each connected component is a path). A
natural question is the following.

\begin{quest}\label{q12}
Is there a constant $c$ such that for any $k\ge c$, any planar graph
has an equitable partition into $k$ induced linear forests?
\end{quest}

It was pointed out to us by Yair Caro that the (outer)planar graph obtained from a large path by adding a universal vertex shows that Question~\ref{q12} has a negative answer.

\end{document}